\documentclass[11pt, reqno]{amsart}
\usepackage{version}
\usepackage{amssymb}
\usepackage{color}
\newtheorem{theorem}{Theorem}[section]
\newtheorem{lemma}[theorem]{Lemma}
\newtheorem{corollary}[theorem]{Corollary}
\newtheorem{prop}[theorem]{Proposition}

\theoremstyle{definition}

\newtheorem{sublemma}[theorem]{Sublemma}

\newcommand{\B}{\mathbb{B}}
\newcommand{\C}{\mathbb{C}}

\newcommand{\N}{\mathbb{N}}

\newcommand{\R}{\mathbb{R}}

\newcommand{\bT}{\mathbb{T}}

\numberwithin{equation}{section}

\begin{document}
\title[Polynomially convex hulls]{Totally real embeddings with prescribed polynomial hulls}
\author[L. Arosio]{Leandro Arosio$^{\dag}$}

\author[E. F. Wold]{Erlend Forn\ae ss Wold$^{\dag\dag}$}
\thanks{$^{\dag\dag}$Supported by NRC grant no. 240569 }
\thanks{$^{\dag}$Supported by the SIR grant ``NEWHOLITE - New methods in holomorphic iteration'' no. RBSI14CFME}
\thanks{Part of this work was done during the international research program "Several Complex Variables and Complex Dynamics"
at the Center for Advanced Study at the Academy of Science and Letters in Oslo during the academic year 2016/2017. }
\address{ L. Arosio: Dipartimento Di Matematica\\
Universit\`{a} di Roma \textquotedblleft Tor Vergata\textquotedblright\ \\
Via Della Ricerca Scientifica 1, 00133 \\
Roma, Italy} \email{arosio@mat.uniroma2.it}
\address{E. F. Wold: Department of Mathematics\\
University of Oslo\\
Postboks 1053 Blindern, NO-0316 Oslo, Norway}\email{erlendfw@math.uio.no}

%
%
\date{\today}

\begin{abstract}
We embed compact $C^\infty$ manifolds into $\mathbb C^n$ as totally real manifolds with prescribed polynomial hulls.  
As a  consequence we show  that any compact $C^\infty$ manifold of dimension $d$ admits a totally real embedding into $\mathbb C^{\lfloor \frac{3d}{2}\rfloor}$
with non-trivial polynomial hull without complex structure.  
\end{abstract}

\maketitle

\section{Introduction}

We will prove the following result.  

\begin{theorem}\label{main}
Let $X\subset\mathbb C^m$ be a totally real $C^\infty$ submanifold of dimension $d>1$ and let {\bf $K\subset X$} be a compact subset.   Let
$M$ be a compact $C^\infty$  manifold of dimension $d$ (possibly with boundary) and assume that there exists a $C^\infty$ embedding $\phi:X\rightarrow M$.  
Then there exists a totally real $C^\infty$ embedding $\psi:M\rightarrow \mathbb C^{m+\ell}$, where $\ell={\max\{\lfloor \frac{3d}{2}\rfloor-m,0\}}$, such that
\begin{itemize}
\item[(i)] $\psi\circ\phi=\mathrm{id}$ on $K$, and 
\item[(ii)] $\widehat{\psi(M)}=\psi(M)\cup\widehat K$.
\end{itemize} 
\end{theorem}
Notice that by a slight abuse of notation we denote the compact set $K\times \{0\} \subset \C^{m+\ell}$ also by the letter $K$.

Our motivation for proving this result is the following recent result by Izzo and Stout.
\begin{theorem}[Izzo--Stout  \cite{IzzoStout}]\label{IS}
Let $M$ be a compact $C^\infty$ surface (possibly with boundary). Then $M$
embeds into $\mathbb C^3$ as a totally real  $C^\infty$ surface $\Sigma$ such that 
$\widehat\Sigma\setminus \Sigma$ is not empty and does not contain any analytic disc.
\end{theorem}
In the case where $M$ is a torus, Theorem \ref{IS} was proved by Izzo, Samuelsson Kalm  and the second named author \cite{ISW}. 
In a recent paper Gupta \cite{gupta} gave an explicit example of an isotropic torus in $\C^3$  whose polynomial hull consists only of an annulus attached to the torus.
We note that the torus embedded in \cite{ISW} is also isotropic, the image being the graph of a real valued function on the torus in $\C^2$.

 As a corollary to Theorem \ref{main} we prove the following generalization of Theorem \ref{IS} to every dimension $d>1$.

\begin{corollary}\label{cor}
Let $M$ be a compact $C^\infty$ manifold (possibly with boundary)  of dimension $d>1$. Then $M$
embeds into $\mathbb C^{\lfloor \frac{3d}{2}\rfloor}$ as a totally real $C^\infty$ submanifold $\Sigma$ such that 
$\widehat\Sigma\setminus \Sigma$ is not empty and does not contain any analytic disc.
\end{corollary}
 
The same result was proved in \cite{ISW} for dimension $2d+4$ of the target space.  Notice that by \cite[Theorem 2.1]{NY} the dimension $\lfloor \frac{3d}{2}\rfloor$ is  optimal  for totally real embeddings of $d$-dimensional manifolds, which implies that Corollary \ref{cor} is optimal.

We emphasize  that Theorem \ref{main} is a simple consequence of the general fact that any 
$C^\infty$ manifold of dimension $d$ embeds into $\mathbb C^{\lfloor \frac{3d}{2}\rfloor}$ as a totally 
real submanifold, and the following general genericity result regarding polynomial hulls and 
totally real embeddings (to be proved in Section \ref{proofs}):

\begin{theorem}\label{nohull}
Let $M$ be a compact $C^\infty$ manifold (possibly with  boundary) of dimension $d<n$ and let $f\colon M\to \C^n$ be a  totally real $C^\infty$ embedding.
Let $K\subset \C^n$ be a compact polynomially convex set. Then for any $k\geq 1$ and for any $\varepsilon>0$ there exists a  totally real  $C^\infty$ embedding  $f_\varepsilon\colon M\to \C^n$ such that
\begin{enumerate}
\item $\|f_\varepsilon-f\|_{C^k(M)}<\varepsilon$,
\item $f_\varepsilon-f=0$  on $f^{-1}(K)$, 
\item $\widehat{K\cup f_\varepsilon(M)}= K\cup f_\varepsilon(M)$.
\end{enumerate}
\end{theorem}

In the case $K=\varnothing$, Theorem \ref{nohull} was proved by Forstneric and Rosay for $d=2$ and $n\geq 3$,
by Forstneric  \cite{Forstneric} for $d\leq \frac{2n}{3}$, and  by L\o w and the second named author \cite{LW} for $d<n$. \\

To obtain Corollary \ref{cor}  it is crucial that Alexander \cite{Alexander} has constructed
a suitable compact set $K$ in the totally real distinguished boundary of the bidisk in $\mathbb C^2$ (see Section \ref{sectionth}).

\section{Proof of Theorem \ref{main} and Corollary \ref{cor}}\label{sectionth}
In this section we prove Theorem \ref{main} and Corollary \ref{cor} assuming Theorem \ref{nohull} which will be proved  in the next section.

We recall first the following result \cite[Lemma 5.3]{ForstnericRosay}. By {\sl generic} we mean open and dense in the Whitney $C^\infty$-topology.
\begin{lemma}\label{fr}
Let $M$ be a compact $C^\infty$ manifold (possibly with boundary) of dimension $d$.  Then a generic $C^\infty$  embedding $\psi:M\rightarrow \C^{\lfloor\frac{3d}{2}\rfloor}$
is totally real.
\end{lemma}

\emph{Proof of Theorem \ref{main}:}

 Let $U'$ be a relatively compact  open subset of $X$ containing $K$. Set $U:=\phi(U')$, and let $\psi_1\colon \phi(X)\rightarrow  \C^m$ be the inverse of $\phi$.
Let $\psi_2\colon M\to \C^m$ be a $C^\infty$  mapping which agrees with $\psi_1$ on $U$. 
Notice that since $2(m+\ell)\geq 2d+1$, by Whitney's theorem (see e.g. \cite[Theorem 2.13]{hirsch}) a generic $C^\infty$  mapping from $M$ to $\C^{m+\ell}$ is an embedding, and since $m+\ell\geq \lfloor \frac{3d}{2}\rfloor$,  by Lemma \ref{fr} a generic $C^\infty$  embedding of $M$ into $\mathbb C^{m+\ell}$ is totally real.  Let  $\psi_3\colon M\to \C^{m+\ell}$ be a small perturbation of $\psi_2$ which is a totally real $C^\infty$  embedding.

 Let $\chi\in C^\infty_0(U)$ with $\chi\equiv 1$  in a neighborhood of $\phi(K)$ and $0\leq \chi\leq 1$.  Then if the perturbation $\psi_3$ was small enough, the mapping  from $M$ to $\C^{m+\ell}$ defined as  $\psi_4:=\chi\cdot\psi_1 + (1-\chi)\psi_3$
is a totally real $C^\infty$  immersion    injective in $U$.    There exists a small $C^\infty$  perturbation $\psi\colon M\to \mathbb C^{m+\ell}$
which is a totally real $C^\infty$  embedding and such that $\psi|_{\phi(K)}=\psi_4|_{\phi(K)}$ (see e.g. \cite[Theorem 2.4.3]{muk} where the approximation  is in $C^0$-norm but the proof can be easily adapted to our case) which implies that  $K\subset\psi(M)$. Notice that $\psi_4|_{\phi(K)}=\psi_1|_{\phi(K)}$, and thus $\psi\circ\phi=\mathrm{id}$ on $K$.
Since $m+\ell>d$, we can apply Theorem \ref{nohull} with $K$ replaced by $\widehat K$ and Theorem \ref{main} follows.   $\hfill\square$

 \vspace{\baselineskip}

To prove Corollary \ref{cor} we need the following result of Alexander:

\begin{theorem}\label{alexander}(Alexander, \cite{Alexander})
The standard torus $\bT^2:= \{(e^{i\vartheta},e^{i\psi})\colon \vartheta,\psi\in \R\}$ in $\C^2$ contains a compact subset $K$ such that $\widehat{K}\setminus K$ is not empty but contains no analytic discs. Such a set can be found in every neighborhood of the diagonal in $\bT^2$.
\end{theorem}

\emph{Proof of Corollary \ref{cor}:}
 Choose a small annular neighbourhood $A$ of the diagonal in $\mathbb T^2$, and 
let $K\subset A$ be the compact subset given by Theorem \ref{alexander}.
Define a totally real $C^\infty$ submanifold $X$ of $\mathbb C^d$
by
$$X=\{(z_1,z_2,z_3,\ldots ,z_d):(z_1,z_2)\in A, \mathrm{Im}(z_j)=0 \mbox{ and} -\varepsilon<\mathrm{Re}(z_j)<\varepsilon \mbox{ for } j=3,\ldots ,d\},$$
 where $\varepsilon>0$.
Now $X$ embeds into the into (any small portion) of the interior of $M$.  By Theorem \ref{main}, there exists a totally real $C^\infty$ embedding $\psi\colon M\to \C^{\lfloor \frac{3d}{2}\rfloor}$ such that 
$K\subset \psi(M)$ and  $\widehat{\psi(M)}=\psi(M)\cup\widehat K$. 
 Recall \cite[Theorem 6.3.2]{stout} that  for a polynomially convex subset $H$ of a totally real $C^\infty$ submanifold the uniform algebra $\mathcal{P}(H)$, consisting of all the functions on $H$ that can be approximated uniformly by polynomials, coincides with the algebra of continuous functions on $H$. Hence  the hull $\widehat K$ cannot be contained in $\psi(M),$ and  the result follows from Theorem \ref{alexander}. $\hfill\square$

\vspace{\baselineskip}

\section{Proof of the main lemma}

For any compact set $K\subset \C^n$ set $h(K):=\overline{\widehat K\setminus K}.$
 We will sometimes denote $\widehat K$ by $[K\widehat]$.

The proof of the following main lemma  is essentially contained in  \cite[Proposition 6,7]{LW} where it is given in the case of $C^1$-regularity. We give a proof for completeness.
\begin{lemma}\label{littlehull}
Let $M$ be a compact $C^\infty$ manifold (possibly with  boundary) of dimension $d<n$ and let $f\colon M\to \C^n$ be a  totally real $C^\infty$ embedding.
Let $K\subset \C^n$ be a compact polynomially convex set, and let $U$ be a neighborhood of $K$. Then for any $k\geq 1$ and for any $\varepsilon>0$ there exists a  totally real $C^\infty$  embedding $f_\varepsilon\colon M\to \C^n$ such that
\begin{enumerate}
\item $\|f_\varepsilon-f\|_{C^k(M)}<\varepsilon$,
\item $f_\varepsilon-f=0$  in a neighborhood of $f^{-1}(K)$, 
\item $\widehat{K\cup f_\varepsilon(M)}\subset U\cup f_\varepsilon(M)$.
\end{enumerate}
\end{lemma}

The following result is proved in \cite[Proposition 4]{LW}.

\begin{prop}\label{proposition4}
Let $M$ be a compact $C^1$ manifold (possibly with  boundary) and let $f\colon M\to \C^n$ be a  totally real $C^1$ embedding.
Let $U'\subset \subset U\subset \subset \C^n$ be open sets. Then there exists an open neighborhood $\Omega$ of $f(M)$ such that
\begin{enumerate}
\item   if $S\subset M$ is closed and $K\subset U'$ is compact, then 
$$ h(K\cup f(S))\subset U'\cup \Omega\Rightarrow h(K\cup f(S))\subset U,$$
\item if $ \tilde f$ is a sufficiently small $C^1$-perturbation of $f$, then (1) holds with   $\tilde f$ in place of $f$.

\end{enumerate}
\end{prop}

The following lemma is proved in \cite[Corollary 2]{LW}
\begin{lemma}\label{corollary2}
Let $K\subset \C^n$ be a  compact subset. Let $M$ be a compact $C^1$ manifold (possibly with boundary) and let $f\colon M\to \C^n$ be a totally real $C^1$ embedding. Let $S\subset M$ be a closed subset.
Let $U$ be an open set such that 
$h(K\cup f(S))\subset U$. Then there exists a constant $c>0$ such that if 
$\|f-\tilde f\|_{C^1(M)}<c$, then  $h(K\cup \tilde f(S))\subset U$.
\end{lemma}

\noindent{\sl Proof of Lemma \ref{littlehull}:} 
The proof is in two steps. Let $I$ denote the interval $[0,1]\subset \R$.\\

\noindent{\sf First Step:} 
The first step is proving the lemma in the case $M=I^d$.
The proof is by induction on $d$.
Fix a strictly positive $d\in \N$, and assume that the result holds for embeddings of $I^{d-1}$.

 Since $K$ is polynomially convex, it admits fundamental system of open neighborhoods which are Runge  and Stein open sets. Thus there exists a Runge  and Stein  open set $U'\subset \C^n$ such that $K\subset U'\subset\subset  U$. Since $U'$  is Runge  and Stein, it admits a normal exhaustion by polynomially convex subsets. Hence there exists a polynomially convex set $K'\subset U'$ such that  $K\subset\mathrm{int}(K')$.
Let us fix some notation. 
\begin{enumerate}
\item for  $J\in \N$ and $\alpha\in \N^d,$ $0\leq \alpha_j\leq J-1$, we denote by $I^J_\alpha$ the cube
$$I^J_\alpha:=\left[ \frac{\alpha_1}{J},\frac{\alpha_1+1}{J}     \right]\times \dots \times  \left[ \frac{\alpha_d}{J},\frac{\alpha_d+1}{J}     \right],$$
\item for $1\leq m\leq d$, $0\leq j\leq J$, we denote by $G^J_{m,j}$ the $(d-1)$-dimensional cube
$$G^J_{m,j}:=I\times \dots \times I\times \underset{\mbox{$m$-th index}}{\left\{   \frac{j}{J}  \right\}} \times I\dots \times I,$$
\item we denote by $G^J$ the grid which is the union of all $(d-1)$-dimensional cubes $$G^J:=\bigcup_{m,j} G^J_{m,j},$$
\item we denote by $G^J(\delta)$ the closed $\delta$-neighborhood of the grid $G^J$,
$$G^J(\delta):=\{x\in I^d\colon {\rm dist}(x,G^J)\leq \delta\},$$
\item we denote by $Q^J_\alpha(\delta)\subset I^J_\alpha$ the smaller $d$-dimensional cube $$Q^J_\alpha(\delta):= \overline{I^J_\alpha\setminus G^J(\delta)},$$
\item  we denote by $S^J$ an $n$-tuple  of distinct cubes $I^J_\alpha.$
\end{enumerate}

We will prove the following statement. There exists a big enough $J$, a small enough $\delta$ and a  totally real $C^\infty$  embedding $f_\varepsilon\colon I^d\to \C^n$ such that
\begin{itemize}
\item[(i)] $\|f_\varepsilon-f\|_{C^k(I^d)}<\varepsilon$,
\item[(ii)]  $f_\varepsilon-f=0$  in a neighborhood of $f^{-1}(K)$, 
\item[(iii)] $ f_\varepsilon(I^J_\alpha)\cap K\neq \varnothing \Rightarrow  f_\varepsilon(I^J_\alpha)\subset K'$,
\item[(iv)] $h(K'\cup f_{\varepsilon}(S^J) \cup f_{\varepsilon}(G^J(\delta )))\subset U$ for any choice of $n$-tuple $S^J$,
\item[(v)] for all $\alpha$ such that $f_\varepsilon(Q_\alpha^J(\delta))\subset \C^n\setminus K,$ there exists an entire function $g_\alpha\in \mathcal{O}(\C^n)$ such that $f_\varepsilon(Q_\alpha^J(\delta))\subset \{g_\alpha=0\}=:Z(g_\alpha),$
\item[(vi)]  for all choices of $n+1$ different multi-indices $\alpha^1,\dots, \alpha^{n+1}$   such that $f_\varepsilon(Q_\alpha^J(\delta))\subset \C^n\setminus K$ we have $\bigcap_{1\leq j\leq n+1 }Z(g_{\alpha^j})=\varnothing$.
\end{itemize}
Once this is proved, the first step of the  proof goes as follows. Assume that $$q\in [K\cup f_{\varepsilon }(I^d)\widehat{ ]}.$$
Let $\mu$ be a representative Jensen measure for the linear functional on $\mathcal{P}(K\cup f_{\varepsilon }(I^d))$  defined as the evaluation at $q$, such that 
for all entire functions $g\in \mathcal{O}(\C^n)$ we have 
\begin{equation}\label{x}
\log |g(q))|\leq \int \log |g| {\rm d}\mu.
\end{equation}
Assume that there exists $\alpha$ such that  $\mu(f_\varepsilon(Q_\alpha^J(\delta)))>0$ and $f_\varepsilon(Q_\alpha^J(\delta))\cap K=\varnothing.$ By (v) we get that 
$\int \log|g_\alpha|{\rm d}\mu=-\infty$, and by (\ref{x}) we have that $q\in Z(g_\alpha)$.
By (vi), there exists an $n$-tuple of cubes $S^J$ such that the measure $\mu$ is concentrated on the subset $$K'\cup f_\varepsilon (S^J)\cup f_\varepsilon (G^J(\delta)).$$
It follows that $q\in [K'\cup f_\varepsilon (S^J)\cup f_\varepsilon (G^J(\delta))\widehat ],$ and by (iv) we obtain that $q\in U\cup f_\varepsilon(I^d)$. \

 We proceed to prove the statement above.

\noindent\emph{Claim 1:}
If $J$ is big enough then 
\begin{equation}\label{base}
h(K'\cup f(S^J))\subset U'
\end{equation} 
for any choice of $n$-tuple $S^J$. Moreover, for a fixed big enough $J$, if $f_0$ is a sufficiently small $C^1$-perturbation of $f$, then (\ref{base}) holds with $f_0$ in place of $f$.

\noindent\emph{Proof of Claim 1:}
Let $U''$ be an open neighborhood of $K'$ with $U''\subset \subset U'$. Let $\Omega'$ be given by Proposition \ref{proposition4} with data $(f,U',U'')$. 
Assume by contradiction that there is a sequence $(J_t)_{t\in \N}$, $J_t\to\infty$ and a sequence $S^{J_t}$ of $n$-tuples such that 
$$h(K'\cup f(S^{J_t}))\not\subset U'.$$
Let $I^{J_t}_{\alpha^1(t)}, \dots, I^{J_t}_{\alpha^n(t)}$ be the cubes in $S^{J_t}$.
Up to passing to a subsequence we may assume that for any $1\leq j\leq n$, the image $f(I^{J_t}_{\alpha^j(t)})$ converges to a point $q_j\in M$ as $t\to\infty$.
Since $K'\cup\{q_j\}_{1\leq j\leq n}$ is polynomially convex, there exists a Runge and Stein neighborhood $V$ of $K'\cup\{q_j\}_{1\leq j\leq n}$ such that $V\subset U''\cup \Omega'$.
For large enough $t$ we have that $K'\cup f(S^{J_t})\subset V$, which implies $h(K'\cup f(S^{J_t}))\subset U''\cup\Omega',$ and thus by  Proposition \ref{proposition4} we have a contradiction. 

Since there are only a finite number of $n$-tuples $S^J$ for a fixed $J$, the perturbation claim follows from Lemma \ref{corollary2}.
\noindent\emph{End proof of Claim 1.}\\

Choose $J$ big enough to ensure that Claim 1 holds and that for all $\alpha,$ $$f(I^J_\alpha)\cap K\neq \varnothing\Rightarrow f(I^J_\alpha)\subset \mathrm{int}(K') .$$

\noindent\emph{Claim 2:}
There exists a totally real $C^\infty$ embedding $f_{\frac{\varepsilon}{2}}\colon I^d\to \C^n$ such that
\begin{enumerate}
\item $\|f_{\frac{\varepsilon}{2}}-f\|_{C^k(I^d)}<\frac{\varepsilon}{2},$ 
\item $f_{\frac{\varepsilon}{2}}=f$  in a neighborhood of $f^{-1}(K')$, 
\item $h(K'\cup f_{\frac{\varepsilon}{2}}(S^J) \cup f_{\frac{\varepsilon}{2}}(G^J))\subset U'$ for any choice of $n$-tuple $S^J$.
\end{enumerate}
Moreover, if $f_0$ is a sufficiently small $C^1$-perturbation of $f_{\frac{\varepsilon}{2}}$, then (3) holds with $f_0$ in place of $f_{\frac{\varepsilon}{2}}$.

\noindent\emph{Proof of Claim 2:}
Fix an $n$-tuple $S^J$. Choose an ordering of the $(d-1)$-dimensional cubes $(G^J_{m,j})_{m,j}$ and denote them by $G_1,\dots ,G_\ell$.

Fix $0<\eta<<1$ (to be determined later). We will construct  inductively a family of totally real $C^\infty$ embeddings $(f_j\colon I^d\to \C^n)_{0\leq j\leq \ell}$  such that 
\begin{itemize}
\item[(a)]$\|f_j-f_{j-1}\|_{C^k(I^d)}\leq  \eta$,
\item[(b)] $f_j=f_{j-1}$  in a neighborhood of $f^{-1}(K')\cup S^J\cup \bigcup_{1\leq i\leq j-1}G_i$,
\item [(c)] $h\left(K'\cup f_j(S^J)\cup \bigcup_{1\leq i\leq j}f_j(G_i)\right)\subset U'.$
\end{itemize}
If we set $f_0:=f$, then (c) is true by Claim 1.
Assume we constructed $f_j$, where $0\leq j\leq \ell-1$.
Then by (c) the set $h(K'\cup f_j(S^J)\cup \bigcup_{0\leq i\leq j}f_j(G_i))$ is compact in $U'$, so there exists a set $U''_1$ such that 
$K'\subset U''_1 \subset \subset U'$ satisfying $h(K'\cup f_j(S^J)\cup \bigcup_{1\leq i\leq j}f_j(G_i))\subset U''_1$.
Let $\Omega_1$ be given by Proposition \ref{proposition4} with data $(f_j, U',U''_1)$.
Consider the  totally real $C^\infty$ embedding $f_j|_{G_{j+1}}\colon G_{j+1}\to \C^n$.
Since $$[  K'\cup f_j(S^J)\cup \bigcup_{0\leq i\leq j}f_j(G_i)\widehat{]}\subset  U''_1\cup \Omega_1,$$
by the inductive assumption  (the result holds for embeddings of $I^{d-1}$) there exists a totally real $C^\infty$ embedding $\tilde f_j|_{G_{j+1}}\colon G_{j+1}\to \C^n$ such that
\begin{enumerate}
\item $\|\tilde f_j-f_j\|_{C^k(G_{j+1})}$ is small, 
\item $\tilde f_j=f_j$  in a neighborhood of $f_j^{-1}(K')\cup S^J\cup \bigcup_{1\leq i\leq j}G_i$,
\item $h\left( K'\cup f_j(S^J)\cup \bigcup_{1\leq i\leq j}f_j(G_i)\cup \tilde f_j(G_{j+1})\right)\subset U''_1\cup \Omega_1$.
\end{enumerate}
If $\|\tilde f_j-f_j\|_{C^k(G_{j+1})}$ is small enough, then we may extend $\tilde f_j$ to a totally real $C^\infty$ embedding $f_{j+1}\colon I^d\to \C^n$ such that
$\| f_{j+1}-f_j\|_{C^k(I^d)}\leq \eta$, which  coincides with $f_j$ on $f^{-1}(K')\cup S^J\cup \bigcup_{1\leq i\leq j}G_i.$
By the choice of $\Omega_1$ it follows that $$h( K'\cup f_{j+1}(S^J)\cup \bigcup_{1\leq i\leq j+1}f_{j+1}(G_i))\subset U'.$$
The mapping $f_\ell$ satisfies Claim 2 for the chosen $n$-tuple $S^J$.

We then iterate this argument for every $n$-tuple $S^J$ (choosing $\eta$ small enough),  and we obtain $f_{\frac{\varepsilon}{2}}$. The perturbation claim follows from Lemma \ref{corollary2}.

\noindent\emph{End proof of Claim 2.}\\

Let $\Omega$ be an open neighborhood of $f_{\frac{\varepsilon}{2}}(I^d)$ given by Proposition \ref{proposition4} with data $(f_{\frac{\varepsilon}{2}},U,U')$.

\noindent\emph{Claim 3:}
There exists a $\delta>0$ such that
\begin{equation}\label{claim3eq}
h(K'\cup f_{\frac{\varepsilon}{2}}(S^J) \cup f_{\frac{\varepsilon}{2}}(G^J(\delta)))\subset U
\end{equation} for any choice of $n$-tuple $S^J$.
Moreover, if $f_0$ is a sufficiently small $C^1$-perturbation of $f_{\frac{\varepsilon}{2}}$, then (\ref{claim3eq}) holds with $f_0$ in place of $f_{\frac{\varepsilon}{2}}$.

\noindent\emph{Proof of Claim 3:}
Fix an $n$-tuple $S^J$. Since $h(K'\cup f_{\frac{\varepsilon}{2}}(S^J) \cup f_{\frac{\varepsilon}{2}}(G^J))\subset U'$, there exists a Runge and Stein neighborhood $V\subset \subset \Omega\cup U'$ of $K'\cup f_{\frac{\varepsilon}{2}}(S^J) \cup f_{\frac{\varepsilon}{2}}(G^J)$. If $\delta>0$ is small enough we have that
$K'\cup f_{\frac{\varepsilon}{2}}(S^J) \cup f_{\frac{\varepsilon}{2}}(G^J(\delta))\subset \subset V$, and thus 
$h(K'\cup f_{\frac{\varepsilon}{2}}(S^J) \cup f_{\frac{\varepsilon}{2}}(G^J(\delta))) \subset U$.
Since the number of $n$-tuples is finite, we can choose a $\delta$ which works for all of them.  The perturbation claim follows from Lemma \ref{corollary2}.
\
 \noindent\emph{End proof of Claim 3.}\\

We end Step 1  by pushing each $f_{\frac{\varepsilon}{2}}(Q^J_\alpha(\delta))$ with $f_{\frac{\varepsilon}{2}}(Q^J_\alpha(\delta))\cap K=\varnothing$ into an analytic hypersurface of $\C^n$.

 Fix a multi-index $\alpha$ such that $f_{\frac{\varepsilon}{2}}(Q^J_\alpha(\delta))\cap K=\varnothing$.
We  identify the cube $I^J_{\alpha}$ with  the cube centred at the origin in $ \mathbb R^d$, with the decomposition $\mathbb R^d\oplus i\mathbb R^d\oplus\mathbb C^{n-d}$ of $\mathbb C^n$, and  
after an affine linear change of coordinates we may assume that $f_{\frac{\varepsilon}{2}}:I^J_\alpha\rightarrow\mathbb C^n$
satisfies $f_{\frac{\varepsilon}{2}}(0)=0$, with the tangent space of $f_{\frac{\varepsilon}{2}}(I^J_\alpha)$ at the origin 
equal to $\mathbb R^d$, and with $Df_{\frac{\varepsilon}{2}}(0)$ orientation preserving.   \

We now claim that we may approximate $f_{\frac{\varepsilon}{2}}$ arbitrarily well in $C^k$-norm on $I^J_\alpha$ by 
holomorphic automorphisms  $G_\alpha$ of $\mathbb C^n$.   Granted this for a moment, we proceed as follows.  
Let $\chi$ be a cutoff function with compact support in $\mathrm{int}(I^J_\alpha)$ such that $\chi\equiv 1$   in a neighborhood of
$Q^J_\alpha(\delta)$.  Define 
\begin{equation}
f_{\alpha}(x):= f_{\frac{\varepsilon}{2}}(x) + \chi(x)(G_\alpha(x)-f_{\frac{\varepsilon}{2}}(x)).
\end{equation} 

If we define $g_\alpha$ as the $n$-th coordinate of $G_\alpha^{-1}$, then clearly $$g_\alpha(f_\alpha(Q_\alpha^J(\delta)))=0.$$
Notice that since $f_\alpha=f_{\frac{\varepsilon}{2}}$ on $I^d\setminus \mathrm {int}(I^J_\alpha)$, the maps $f_\alpha$ patch up as $\alpha$ varies among all  multi-indices $\alpha$ such that $f_{\frac{\varepsilon}{2}}(Q^J_\alpha(\delta))\cap K=\varnothing$, and we obtain an embedding $f_\varepsilon\colon I^d\to \C^n$ satisfying the properties (i)-(v).

{

Finally, by the following sublemma, we may assume, possibly having to perturb the automorphisms $G_\alpha$ slightly,  that the intersection $\bigcap_{1\leq j\leq n+1 }Z(g_{\alpha^j})$ of any collection of $n+1$ zero sets is empty. Thus property (vi) is satisfied.

\begin{sublemma}
Let $\{g_j\}_{1\leq j\leq N}\subset\mathcal O(\mathbb C^n)$ be a finite collection of non-constant holomorphic functions, and for 
each $a_j\in\mathbb C$ set $g_j^{a_j}:=g_j-a_j$.  Then there exists a dense $G_\delta$ set $A\subset\mathbb C^N$
such that for each $a\in A$ the following holds: for any collection $\{g_{i_1}^{a_{i_1}},...,g_{i_{n+1}}^{a_{i_{n+1}}}\}$ with 
$i_j\neq i_k$ if $j\neq k$, we have that 
\begin{equation}\label{intersection}
\bigcap_{l=1}^{n+1}Z(g_{i_l}^{a_{i_l}})=\varnothing.  
\end{equation}
\end{sublemma}
\begin{proof}


Let $a=(a_1,....,a_N)\in\mathbb C^N$ be an arbitrary point.   Let $\{I_i\}_{i=1}^M=\{(i_1,...,i_{n+1})\}_{i=1}^M$ be the collection 
of vectors with all possible combinations, without repetition, of $n+1$ of the indices $\{1,....,N\}$.  Let $\delta>0$ and fix $R>0$. 
We consider first the multi-index $I_1$.   The set $Z(g^{a_{1_1}}_{1_1})\cap\mathbb B^n_{R}$ consists 
of a finite number of irreducible components, each of dimension $n-1$.   
Denote by $\triangle_{\delta}(p)\subset \C$ the disc of center $p$ and radius $\delta $, and by $\B^n_R\subset \C^n$ the ball of radius $R$ centered at the origin. Choosing $a^1_{1_2}\in\triangle_{\delta}(a_{1_2})$
such that $g^{a^1_{1_2}}_{1_2}$ is not identically zero on the regular part of any of these components, we get that 
\begin{equation}
Z(g^{a^1_{1_1}}_{1_1})\cap Z(g^{a^1_{1_2}}_{1_2})\cap\mathbb B^n_{R}
\end{equation}
has a finite number of irreducible components, each of dimension at most $n-2$, where we have set $a_{1_1}^1=a_{1_1}$.
Choosing $a^1_{1_3}\in\triangle_\delta(a_{1_3})$ such that $g^{a^1_{1_3}}_{1_3}$ is not identically zero 
on any of these components, and continuing in this fashion, we obtain a collection of points $\{a^1_{1_l}\}, l=1,...,n+1,$
such that 
\begin{equation}\label{isect}
\bigcap_{l=1}^{n+1}Z(g_{1_l}^{a^1_{1_l}})\cap\mathbb B^n_{R}=\varnothing, 
\end{equation}
and such that $a^1_{1_l}\subset\triangle_\delta(a_{1_l})$.  Finally, by continuity, \eqref{isect}
is an open condition on the constants $a^1_{1_l}$, so we may choose $\delta_1>0$
such that $ \triangle_{\delta_1}(a^1_{1_l})\subset\triangle_\delta(a_{1_l})$ such that \eqref{isect}
holds with each $a^1_{1_l}$ replaced by an arbitrary $\tilde a_{1_l}\in\triangle_{\delta_1}(a^1_{1_l})$. \

Based on this we continue an inductive construction, and we
assume that we have constructed points $\{a_{i_l}^k\}$ for $i=1,...,k$ and $l=1,...,n+1$, and $\delta_k>0$, such
that $\triangle_{\delta_k}(a^k_{i_l})\subset \triangle_\delta(a_{i_l})$, and such that 
\begin{equation}
\bigcap_{l=1}^{n+1}Z(g_{i_l}^{\tilde a_{i_l}})\cap\mathbb B^n_R=\varnothing,  
\end{equation}
for $i=1,...,k$, and $\tilde a_{i_l}\subset \triangle_{\delta_k}(a^k_{i_l})$.  Note that points may appear with repetition 
in the collection $\{a^k_{i_l}\}$.  \

We now consider the multi-index $\{(k+1)_1,...,(k+1)_{n+1}\}$.   As we have seen, the condition corresponding to \eqref{isect}
is generic for a single set of $n+1$ functions.  So we may choose points $a_{{(k+1)}_l}^{k+1}$
such that $a_{{(k+1)}_l}^{k+1}\subset \triangle_{\delta_k}(a^k_{{(k+1)}_l})$ if the index $(k+1)_l$
appeared already in the construction, and $a_{{(k+1)}_l}^{k+1}\subset \triangle_{\delta}(a_{{(k+1)}_l})$
if the index did not appear yet, and such that 
\begin{equation}
 \bigcap_{l=1}^{n+1}Z(g_{(k+1)_l}^{a^{k+1}_{(k+1)_l}})\cap\mathbb B^n_{R}=\varnothing.
\end{equation}
 Relabel the remaining points $a_{i_l}^k, i=1,...,k, l=1,...,n+1$, with a superscript $k+1$.
By continuity, we may choose $\delta_{k+1}$ to reproduce our induction hypothesis with $k+1$ in place of $k$.  \

We conclude that for each $R\in\mathbb N$ there exists an open dense set $A_R\subset\mathbb C^N$
such that
\begin{equation}
\bigcap_{l=1}^{n+1}Z(g_{i_l}^{a_{i_l}}) \cap \mathbb B^n_R=\varnothing  
\end{equation}
for all $a\in A_R$ and all $(i_1,...,i_{n+1})$ without repetition.  Set $A=\cap_R A_R$.

\end{proof}

To produce $G_\alpha$ it suffices by  \cite[Main Theorem]{Forstneric} to produce a piecewise smooth isotopy of embeddings   $(G_t\colon I^J_\alpha\to \C^n)_{t\in [0,1]}$ connecting the identity map  $G_0={\sf id}_{I^J_\alpha}$ with the map   $G_1=f_{\frac{\varepsilon}{2}}$, such that $G_t(I^J_\alpha)$ is totally real and polynomially 
convex for all $t$. 
 
 Notice that, if $J$ is big enough, then  $f(I^J_\alpha)$ is polynomially convex  \cite[Proposition 4.2]{MWO}. Once $J$ is fixed, it follows by \cite[Corollary 1]{LW} that if $ f_{\frac{\varepsilon}{2}}$ is a small enough $C^1$-perturbation of $f$, then  $ f_{\frac{\varepsilon}{2}}(I^J_\alpha)$ is polynomially convex. 
Define an isotopy
$f_{\frac{\varepsilon}{2}}((1-t)x)$ for $t\in [0,1-\eta]$, where  $\eta>0$ is to be chosen. Then   $f_{\frac{\varepsilon}{2}}(t I^J_\alpha)$ is totally real and polynomially convex for all $t$. 

 Since any sufficiently small $C^1$-perturbation of $I^J_\alpha$ is polynomially convex and totally real,  if $\eta>0$ is small enough, we can choose a neighborhood $W$ of the origin in  $f_{\frac{\varepsilon}{2}}(I^J_\alpha)$ and an isotopy of embeddings $(\pi_t\colon W\to \C^d)_{t\in [0,1]}$ 
such that 
\begin{enumerate}
\item $f_{\frac{\varepsilon}{2}}(\eta I^J_\alpha)\subset W$,
\item $\pi_0={\sf id}_{W}$ and $\pi_1$ is the projection to $\mathbb R^d$,
\item $\pi_t (f_{\frac{\varepsilon}{2}}(\eta I^J_\alpha))$ is totally real and polynomially convex for all $t\in [0,1].$
\end{enumerate} 

We have  thus connected $f_{\frac{\varepsilon}{2}}$ to an orientation preserving embedding $H:I^J_\alpha\rightarrow\mathbb R^d$
with $H(0)=0$.  Now $H$ is connected to $DH(0)$ via the path $\frac{1}{t}H(tx)$, and then 
$DH(0)$ is connected to the identity map  ${\sf id}_{I^J_\alpha}$.

\noindent{\sf Second Step:} 

For simplicity we assume that $M$ is without boundary.  
Let $I_1,...,I_N$ be a cover of $M$ by closed cubes, such that there is a collection of subcubes $I_i^0\subset\subset \mathrm{int}(I_i)$
which also cover $M$.  For each $i$  let $\chi_i$ be a cutoff function compactly supported in $\mathrm{int}(I_i)$ with 
$\chi_i\equiv 1$  in a neighborhood of $I_i^0$.   Let $U_j\subset\subset U_{j+1}\subset U$ be open sets for $j=1,...,N$, and let 
$\Omega$ be given by Proposition \ref{proposition4}  for all the  data $(f(M),U_{j+1},U_j)$  with $j=1,...,N$.
We will proceed to perturb the image $f(M)$ cube by cube.\

Assume that we have constructed a small $C^k$-perturbation $f_j:M\rightarrow\mathbb C^n$ of $f$ such that 
$$
h(K\cup f_j(\cup_{i=1}^j I^0_i))\subset U_j
$$
(it will be clear from the construction how to construct $f_1$). 

By Step 1 we let $g_{j+1}:I_{j+1}\rightarrow\mathbb C^n$ be a small $C^k$-perturbation of $f_j$ such that 
$$
h(K\cup f_j(\cup_{i=1}^j I^0_i)\cup g_{j+1}(I_{j+1}))\subset \Omega\cup U_j,
$$
and $g_{j+1}$ coincides with $f_j$ in a neighbourhood of  $\cup_{i=1}^j I_i^0\cup f_j^{-1}(K)$ .
Defining $f_{j+1}$ by
$$
f_{j+1} := f_j + \chi_{j+1}(g_{j+1}-f_j),
$$
 we obtain $$h(K\cup f_{j+1}(\cup_{i=1}^{j+1} I^0_i))\subset U_{j+1}.$$
 If all the perturbations were small enough, defining $f_\varepsilon:=f_{N}$ yields the result.
$\hfill\square$

\section{Proof of Theorem \ref{nohull}}\label{proofs}

 Let $\varepsilon>0$. Let $(U_j)_{j\geq 1}$ be a sequence of open neighborhoods of $K$ such that $U_{j+1}\subset U_{j}$ for all $j\geq 1$ and $K=\bigcap_{j\geq 1}U_j.$ Set $f_0:= f$.
We perturb $f_0$ inductively using Lemma \ref{littlehull}, obtaining a family of totally real $C^\infty$ embeddings $(f_j\colon M\to \C^n)_{j\in \N}$ such that,
\begin{enumerate}
\item for all $j\geq 1$ we have $\widehat{K\cup f_j(M)}\subset U_j\cup f_j(M),$ 
\item for all $j\in \N$ we have  $$\|f_{j+1}-f_j\|_{C^{k+j}(M)}< \eta_j,$$ where the sequence $(\eta_j)$ satisfies for all $\ell\in \N$, $$\sum_{j=\ell}^\infty \eta_j<\delta_\ell,$$ where  $0<\delta_0\leq \varepsilon$ is to be chosen  and for all $\ell\geq 1$, $\delta_\ell$ 
  is the constant  $c(f_\ell,K, U_\ell)$ given by Lemma \ref{corollary2},
\item for all $j\in \N$ we have $f_{j+1}\equiv f_j$ on $f^{-1}(K)$.
\end{enumerate}
The sequence $(f_j)$ clearly converges to a $C^\infty$ map $f_\varepsilon \colon M\to \C^n$ which coincides with $f$ on $f^{-1}(K)$, and such that 
$\|f-f_\varepsilon\|_{C^k(M)}<\delta_0$.
 If $\delta_0$ is small enough, then $f_\varepsilon$ is a totally real $C^\infty$ embedding. We claim that for all $\ell\geq 1$, $$\widehat{K\cup f_\varepsilon(M)}\subset U_\ell\cup f_\varepsilon(M).$$ This follows from Lemma \ref{corollary2}, since $$\|f_\varepsilon-f_\ell\|_{C^1(M)}\leq  \sum_{j=\ell}^\infty \eta_j<\delta_\ell.$$ $\hfill\square$

\bibliographystyle{amsplain}

\begin{thebibliography}{10}

\bibitem{Alexander}
H. Alexander,
{\sl Hulls of subsets of the torus in $\mathbb C^2$},  
Ann. Inst. Fourier (Grenoble) {\bf 48} (1998), no. 3, 785--795. 


\bibitem{Forstneric}
F. Forstneric, 
{\sl Approximation by automorphisms on smooth submanifolds of $\mathbb C^n$}, 
Math. Ann. 300 (1994), 719--738.


\bibitem{ForstnericRosay}
F. Forstneric,  J.-P. Rosay, 
{\sl Approximation of biholomorphic mappings by automorphisms of $\mathbb C^n$}, Invent. Math. {\bf 112} (1993), no. 2, 323--349.


\bibitem{gupta}
P. Gupta, {\sl A nonpolynomially convex isotropic two-torus with no attached discs}, arXiv:1610.09356.



 \bibitem{hirsch}
M. W. Hirsch, {\sl Differential Topology}, Graduate Texts in Mathematics {\bf 33} (1976).


 \bibitem{muk}
A. Mukherjee, {\sl Differential Topology}, Birkh\"auser Basel (2015).


\bibitem{NY}
P. T. Ho, H. Jacobowitz, P. Landweber, {\sl Optimality for totally real immersions and independent mappings of manifolds into $\C^N$},
New York J. Math. {\bf 18} (2012) 463--477. 

\bibitem{ISW}
A. Izzo, H. Samuelsson Kalm, E. F. Wold,  {\sl Presence or absence of analytic structure in maximal ideal spaces}, Math. Ann. {\bf 366} (2016), no. 1-2, 459--478.




\bibitem{IzzoStout}
A. Izzo, E. L. Stout, {\sl Hulls of surfaces},  Indiana U. Math. J. (to appear).



\bibitem{LW} E. L\o w, E. F. Wold, {\sl Polynomial convexity and totally real manifolds}, Complex Var. Elliptic 
{\bf 54} (2009),  Nos. 3--4,  265--281.

 \bibitem{MWO}
P. E. Manne, E. F. Wold, N. \O vrelid, {\sl Holomorphic convexity and Carleman approximation by entire functions on Stein manifolds}, Math. Ann {\bf 351} (2011), 571--585.




\bibitem{stout}
E. L. Stout, {\sl Polynomial convexity}, Progress in Mathematics, {\bf 261} Birkh\"auser Boston, Inc., Boston, MA, 2007.



\end{thebibliography}

\end{document}